\documentclass[EJP,preprint]{ejpecp} 

\usepackage{tabularx}
\usepackage{booktabs}

\SHORTTITLE{The Frog Model on $\mathbb{Z}$ with General Random Survival Parameter}

\TITLE{The Frog Model on $\mathbb{Z}$ with General Random Survival Parameter\support{Supported
    by the Institute of Mathematical Statistics (IMS) and the Bernoulli
    Society.}\
    \thanks{Current maintainer of class file is
      \href{https://vtex.lt}{VTeX, Lithuania}. Please send all queries to
      \href{mailto:latex-support@vtex.lt}{\texttt{latex-support@vtex.lt}}.}} 

\AUTHORS{%
  Gustavo~O. Carvalho\footnote{Universidade de São Paulo, Brazil.\EMAIL{gustavoodc@ime.usp.br}}
  \and 
  Fábio~P. Machado\footnote{Universidade de São Paulo, Brazil.\EMAIL{fmachado@ime.usp.br}}
  \and J. Hermenegildo~R. González\footnote{Universidade de São Paulo, Brazil.\EMAIL{hermenegildo.ramirez@usp.br}} 
   }

\KEYWORDS{Frog model; interacting particle system; phase transition; regularly varying tails}

\AMSSUBJ{60K35; 60K37; 60G50; 60J80}

\SUBMITTED{December 90, 2025} 
\ACCEPTED{} 

\VOLUME{}
\YEAR{}
\PAPERNUM{0}
\DOI{}

\ABSTRACT{We study the frog model on $\mathbb{Z}$ with particle-wise random geometric lifetimes: each particle has a survival parameter $\pi\in(0,1)$ sampled i.i.d., whose density near $1$ satisfies $f_\pi(u)\sim (1-u)^{\beta-1}L\!\big((1-u)^{-1}\big)$ with $\beta>0$, and $L$ slowly varying. This strictly extends the $\mathrm{Beta}(\alpha,\beta)$ case. Let $\eta$ denote the common law of the i.i.d.\ initial number of particles $\{\eta_x\}_{x\in\mathbb{Z}}$. Using a percolation comparison and sharp one-particle displacement tails, we obtain a universal threshold at $\beta=\tfrac12$. If $\beta>\tfrac12$ and $E(\eta)<\infty$, extinction occurs almost surely. If $\beta<\tfrac12$ and $\mathbb{P}(\eta=0)<1$, survival has positive probability. At the boundary $\beta=\tfrac12$ we give sharp criteria: extinction if $E(\eta)<\infty$ and $8\,\limsup_{n\to\infty}L(n^2)<1/E(\eta)$; survival if $\mathbb{P}(\eta=0)<1$ and $\sqrt{2}\,\liminf_{n\to\infty}L(n^2)>1/E(\eta)$. These results recover the Carvalho–Machado threshold for Beta laws and show that only the exponent $\beta$ governs the phase transition, while $L$ impacts the critical regime.}

\begin{document}

\section{Introduction}

Interacting particle systems that combine random motion with local activation rules have provided deep insights into the probabilistic mechanisms underlying epidemic propagation, rumor dynamics, and branching structures in random environments. Among these, the \emph{frog model}, introduced in the late 1990s and subsequently developed by Alves, Machado and Popov \cite{Alves2002}, has become one of the most studied frameworks. The dynamics are straightforward: particles sit on the vertices of a connected graph; those at the root begin active while all others are inactive. Active particles perform independent simple symmetric random walks, and any inactive particle encountered is immediately activated, itself beginning a random walk and perpetuating the mechanism. The central question is whether the system \emph{survives}, meaning that active particles persist indefinitely, or whether it eventually \emph{dies out}. This binary outcome masks a remarkable richness of behavior, shaped by the geometry of the graph, the distribution of initial particles, and the stochastic laws governing activation and death.

The model exhibits contrasting behaviors across different settings. On trees and higher-dimensional lattices, it has been shown that the frog model undergoes genuine phase transitions, with survival depending on critical parameters that can often be computed or bounded \cite{Fontes2004,Lebensztayn2005,Gallo2023}. On $\mathbb{Z}^d$ with $d \geq 2$, phase transitions occur depending on the number of initial particles per site \cite{Alves2002}. On homogeneous trees, upper and lower bounds for the critical probability have been refined in a sequence of works \cite{Lebensztayn2005,Lebensztayn2019,Lebensztayn2020}. These results established the frog model as a fertile ground for investigating survival phenomena analogous to those of percolation and branching processes. In sharp contrast, on $\mathbb{Z}$ with geometric lifetimes governed by a \emph{fixed} parameter $p$, extinction occurs almost surely under mild assumptions, and no non-trivial survival is possible. This rigidity has motivated the study of variations designed to recover survival, including models with drifted random walks or spatially varying survival parameters \cite{Bertacchi2014,Gantert2009}.

A breakthrough in this direction was recently provided by Carvalho and Machado \cite{CarvalhoMachado2025}, who considered a version of the frog model on $\mathbb{Z}$ with a \emph{random survival parameter}. In their formulation, each particle has a survival probability $\pi$ drawn independently from a distribution on $(0,1)$, and its lifetime is geometric with parameter determined by this realization. Focusing on the Beta family $\pi\sim \mathrm{Beta}(\alpha,\beta)$, they proved a sharp threshold at $\beta=\tfrac12$: extinction for $\beta>1/2$ and survival with positive probability for $\beta<1/2$, with the behavior at the boundary case $\beta=1/2$ depending on $\alpha$ and on the initial configuration.

The present work extends this phenomenon to a broad and natural class of survival distributions whose densities satisfy, near the right endpoint,

\begin{equation} \label{eq:edge-behavior-v1}
f_\pi(u)\ \sim\ (1-u)^{\beta-1}\,L\!\Big(\frac{1}{1-u}\Big)\qquad (u\uparrow 1),
\end{equation}

where $\beta>0$ and $L$ is slowly varying at $\infty$. This class contains the Beta family as a special case ($L\equiv\text{constant}$) and also accommodates a wide range of heterogeneous right tails. We show that the extinction–survival dichotomy depends only on the exponent $\beta$: if $\beta>\tfrac12$ the process dies out, while if $\beta<\tfrac12$ it survives with positive probability. In the critical regime $\beta=\tfrac12$, the outcome hinges on the quantities $\liminf_{n\to\infty}L(n^{2})$ and $\limsup_{n\to\infty}L(n^{2})$, and on the initial configuration.


The remainder of the paper is organized as follows. Section~\ref{sec:def_results} introduce some definitions for the model and states the main results, whose proofs are presented in Section~\ref{Proof}. Section~\ref{Examples} is devoted to examples: Section~\ref{subs:ex1} presents concrete families of survival–parameter laws satisfying condition~\eqref{eq:edge-behavior-v1}, together with their parameter \(\beta\) and slowly varying factor \(L\); Section~\ref{subs:ex2} provides an example of a case that falls in the inconclusive regime of the main theorem, where we can still prove results regarding survival; Section~\ref{subs:ex3} gives an extensive list of distributions for odds satisfying \eqref{limw}.

\section{Definitions and main results}\label{sec:def_results}

This section introduces the formal setup of the frog model on $\mathbb{Z}$ with particle-wise random survival parameters. We specify the initial configuration, the independent random walks, and the geometric lifetimes determined by a common random variable $\pi \in (0,1)$. The analysis focuses on survival-parameter laws whose densities satisfy condition~(\ref{eq:edge-behavior-v1}),
where $\beta > 0$ and $L$ is slowly varying at infinity. This class strictly extends the Beta family and captures a broad range of regularly varying right tails.

We then introduce the notions of survival and extinction of the system. The main result establishes a sharp phase transition governed solely by the exponent~$\beta$. When $\beta > 1/2$, the process becomes extinct almost surely under mild moment assumptions on the initial particle counts, whereas for $\beta < 1/2$ the system survives with positive probability whenever the initial configuration is not identically zero. The critical case $\beta = 1/2$ requires a refined analysis involving the limits of $L(n^{2})$ and the expectation of the initial distribution. These definitions and asymptotic conditions form the basis for the proofs presented in Section~3.

\subsection{Model and notation}

Let $\mathbb{N}=\{1,2,3,\dots\}$ and $\mathbb{N}_0=\mathbb{N}\cup\{0\}$. We say that a function $L:(0,+\infty)\to (0,+\infty)$ is \textit{regularly varying} with index $\rho$ at infinity if
\[\lim_{x\to\infty}\frac{L(\lambda x)}{L(x)}=\lambda^\rho, \qquad \text{for every }\lambda>0;\]
moreover, $L$ is called \textit{slowly varying} at infinity if $\rho=0$. We say that $f(x)\sim g(x)$ as $x\to \infty$ when $\lim_{x\to \infty}f(x)/g(x)=1$; the notation is used analogously for other limits, such as $x\downarrow 0$ or $x\uparrow 1$.

We now define the frog model on \(\mathbb{Z}\) in a formal way. For each $x\in\mathbb{Z}$, let $\eta_x\in\mathbb{N}_0$ be the initial number of particles at $x$. Particles are indexed by $(x,i)$ with $1\le i\le \eta_x$.
The families
\[
\{\eta_x\}_{x\in\mathbb{Z}},\quad \{\pi_{x,i}\}_{x\in\mathbb{Z},\,i\in\mathbb{N}},\quad
\{(S^{x,i}_n)_{n\in\mathbb{N}_0}\}_{x\in\mathbb{Z},\,i\in\mathbb{N}}
\]
are mutually independent, where $\{\eta_x\}_{x\in\mathbb{Z}}$ and $\{\pi_{x,i}\}_{x\in\mathbb{Z},\,i\in\mathbb{N}}$ are i.i.d.\ sequences, and $\{(S^{x,i}_n)_{n\in\mathbb{N}_0}\}_{x,i}$ is a family of independent simple symmetric random walks on $\mathbb{Z}$ with $S^{x,i}_0=x$; the variables $\pi_{x,i}\in(0,1)$ have common law $\pi$ (with density $f_\pi$ when $\pi$ is absolutely continuous).

Given $\pi_{x,i}=p$, the lifetime $L_{x,i}\in\mathbb{N}_0$ is geometric with tail
\[
\mathbb{P}(L_{x,i}\ge k\mid \pi_{x,i}=p)=p^{k},\qquad k\in\mathbb{N}_0.
\]
Conditional on $\{\pi_{x,i}\}$, the lifetimes $\{L_{x,i}\}$ are independent of $\{\eta_x\}_{x\in\mathbb{Z}}$ and of the walks
$\{(S^{x,i}_n)_{n\in\mathbb{N}_0}\}_{x\in\mathbb{Z},\,i\in\mathbb{N}}$ (and of each other).

At time $0$ all particles at the origin are \emph{active} and all others are \emph{inactive}. Each active particle evolves as follows: before each step it dies with probability $1-\pi_{x,i}$; if it survives, it takes one step of its simple symmetric random walk. Whenever an active particle visits a site, all particles
present there become active and evolve independently by the same rules. We denote this system by $\mathrm{FM}(\mathbb{Z},\pi,\eta)$, where $\pi$ is the common law of the survival parameters and $\eta$ is the law of the initial number of particles per vertex.

\subsection{Survival and extinction}

This section formalizes the notions of survival and extinction for the frog model on $\mathbb{Z}$. We work under the framework introduced in Section~2.1, where particle lifetimes are determined by a survival parameter with density exhibiting the asymptotic form $(1-u)^{\beta-1} L((1-u)^{-1})$ as $u \uparrow 1$.

The main theorem presented here establishes a sharp dichotomy governed by the exponent $\beta$. When $\mathbb{E}[\eta] < \infty$ and $\beta > 1/2$, extinction occurs almost surely. In contrast, if $\beta < 1/2$ and the initial configuration is not almost surely empty, the process survives with positive probability. The critical case $\beta = 1/2$ requires additional conditions on the slowly varying function $L(n^{2})$ and the expected number of initial particles. 

\begin{definition}
A realization of $\mathrm{FM}(\mathbb{Z},\pi,\eta)$ \emph{survives} if at all times there exists at least one active particle; otherwise it \emph{dies out}.
\end{definition}

\begin{theorem}\label{teomain}
Assume that the survival-parameter law $\pi$ has density satisfying
condition
\begin{equation} \label{eq:edge-behavior}
f_\pi(u)\ \sim\ (1-u)^{\beta-1}\,L\!\Big(\frac{1}{1-u}\Big)\qquad (u\uparrow 1),
\end{equation}
with $\beta>0$ and $L$ slowly varying at $\infty$. Then:

\smallskip
\noindent (i) If $E(\eta)<\infty$ and either
\begin{itemize}
\item $\beta>\tfrac12$, \quad or
\item $\beta=\tfrac12$ \quad and \quad $\displaystyle 8\;\limsup_{n\to\infty} L(n^{2})\ <\ \frac{1}{E(\eta)}$,
\end{itemize}
then
\[
\mathbb{P}\big[\mathrm{FM}(\mathbb{Z},\pi,\eta)\text{ survives}\big]=0.
\]

\smallskip
\noindent (ii) If $\mathbb{P}(\eta=0)<1$ and either
\begin{itemize}
\item $\beta<\tfrac12$,\quad or
\item $\beta=\tfrac12$ \quad and \quad $\displaystyle \sqrt{2}\;\liminf_{n\to\infty} L(n^{2})\ >\ \frac{1}{E(\eta)}$ \ \textup{(with $1/E(\eta):=0$ when $E(\eta)=\infty$)},
\end{itemize}
then
\[
\mathbb{P}\big[\mathrm{FM}(\mathbb{Z},\pi,\eta)\text{ survives}\big]>0
\]

\end{theorem}

\begin{remark} In particular, $\pi\sim\mathrm{Beta}(\alpha,\beta)$ satisfies condition~\eqref{eq:edge-behavior}, with $L\equiv 1/B(\alpha,\beta)$. When $\beta=\tfrac12$, item \textup{(i)} becomes $B(\alpha,\tfrac12)>8\,E(\eta)$ and item \textup{(ii)} becomes $B(\alpha,\tfrac12)<\sqrt{2}\,E(\eta)$. This improves the results in \cite{CarvalhoMachado2025}.
\end{remark}

\begin{remark}\label{rem1} The following consequences are immediate from the proof of Theorem~\ref{teomain}.
Let $U\in(0,1)$, $\beta>0$, and $L$ be a slowly varying function.
\begin{itemize}
\item If \(f_\pi(u)\ \le\ (1-u)^{\beta-1}\,L\!\Big(\tfrac{1}{1-u}\Big)\quad\text{for all }u\in[U,1),
\) then Theorem~\ref{teomain}\,(i) holds.
\item If \(f_\pi(u)\ \ge\ (1-u)^{\beta-1}\,L\!\Big(\tfrac{1}{1-u}\Big)\quad\text{for all }u\in[U,1),
\) then Theorem~\ref{teomain}\,(ii) holds.
\end{itemize}
\end{remark}







\subsection{Heavy-tailed odds}\label{infinity_law}

Recall that \(\pi\) can be interpreted as the (random) probability that a particle survives at any given instant of time. In this section, we shift our focus from probabilities to odds by setting
\[X:=\frac{\pi}{1-\pi}\in(0,\infty),\]
and concentrate on cases where \(X\) is heavy-tailed. The inverse relation is then given by
\[\pi=\frac{X}{1+X}\in(0,1).\]

\begin{proposition}
\label{proposition:cmp}
Let X be a non-negative random variable with density $f_X(t)$ such that
\begin{equation}\label{limw}
f_X(t)\ \sim\,t^{-(\beta+1)}L_X(t)\qquad (t\rightarrow \infty)
\end{equation}
for some constants \(\beta>0\) and a function $L_X$ slowly varying at infinity.
Then, the random variable \(\pi:=\frac{X}{1+X}\) has a density and \[f_\pi(u)\ \sim(1-u)^{\beta-1}L_X \Big(\frac{1}{1-u}\Big) \qquad (u\uparrow 1).\]

\end{proposition}

Therefore, Proposition~\ref{proposition:cmp} complements Theorem~\ref{teomain} by providing results on survival or extinction through the odds distribution. Moreover, one can also note that condition \eqref{limw} is equivalent to 
\(f_X\) being regularly varying with index \(\rho=-(\beta+1)\) at infinity.

We now state a transfer principle showing that regularly varying tails are stable under multiplication by an independent factor possessing a moment slightly above the tail index. Specifically, if
\(f_X(t)\sim t^{-(\beta+1)}L_X(t)\) as \(t\to\infty\) and \(Y\ge 0\) is absolutely continuous and independent of \(X\) with
\(\mathbb{E}[Y^{\beta+\varepsilon}]<\infty\) for some \(\varepsilon>0\), then the product \(Z:=XY\) retains the same tail index \(\beta\) and carries over the slowly varying factor \(L_X\). The precise statement is as follows.

\begin{proposition}\label{thm:B}
Let \(X\) have density \(f_X\) satisfying
\begin{equation}\label{eq:fX-RV}
f_X(t)\ \sim t^{-(\beta+1)}\,L_X(t)\qquad (t\to\infty),
\end{equation}
for some \(\beta>0\) and \(L_X\) slowly varying at \(\infty\).
Let \(Y\ge 0\) be absolutely continuous and \emph{independent} of \(X\) and suppose there exists \(\varepsilon>0\) such that
\(\mathbb{E}[Y^{\beta+\varepsilon}]<\infty\).
Define \(Z:=XY\). Then, as \(z\to\infty\),
\begin{equation}\label{eq:fZ-asymp}
f_Z(z)\ \sim z^{-(\beta+1)}\,L_Z(z),
\end{equation}
with slowly varying factor
\[L_Z(z):=\mathbb{E}[Y^{\beta}]\,L_X(z)\]
\end{proposition}

\medskip
\noindent

\section{Proofs}\label{Proof} 

This chapter contains the proofs of the main results stated in Section~2. The analysis relies on sharp asymptotic estimates for the maximal displacement of a single particle during its lifetime, together with a percolation-style comparison that translates these estimates into survival or extinction of the full system.

\subsection{Survival and extinction}

For a particle $i$ initially at $z\in\mathbb{Z}$, define the maximal one–sided and two–sided displacements during its lifetime:
\[
D^{\rightarrow}_{z,i}:=\max_{0\le n\le L_{z,i}}\big(S^{z,i}_n-z\big),\qquad
D^{\leftarrow}_{z,i}:=\max_{0\le n\le L_{z,i}}\big(z-S^{z,i}_n\big),\,
\]
\[D^{*}_{z,i}:=\max_{0\le n\le L_{z,i}}\big(|S^{z,i}_n-z|\big)=D^{\rightarrow}_{z,i}\vee D^{\leftarrow}_{z,i}.\]
By homogeneity, write $D^{\rightarrow}$ and $D^{*}$ for generic copies. A key identity (see \cite[Lemma~2.1]{Lebensztayn2005} below) is the exact conditional tail
\begin{equation}\label{eq:tail-right}
\mathbb{P}\!\big(D^{\rightarrow}\ge n\mid \pi=p\big)=\Big(\frac{1-\sqrt{1-p^{2}}}{p}\Big)^{\!n},\qquad n\in\mathbb{N},
\end{equation}
and, by symmetry,
\begin{equation}\label{eq:dstar-half}
\mathbb{P}(D^{\rightarrow}\ge n)=\mathbb{P}(D^{\leftarrow}\ge n)\ \ge\ \tfrac12\,\mathbb{P}(D^{*}\ge n),\qquad n\in\mathbb{N}.
\end{equation}


We rely on a percolation comparison that translates survival/extinction into sharp asymptotics for $D^{\rightarrow}$ and $D^{*}$. We will use Proposition~1.2 in \cite{CarvalhoMachado2025}, which states that

If $E(\eta)<\infty$ and
\[
\limsup_{n\to\infty} n\,\mathbb{P}(D^{*}\ge n)\le 2\limsup_{n\to\infty} n\,\mathbb{P}(D^{\rightarrow}\ge n)\ <\ \frac{1}{2E(\eta)},
\]
then $\mathbb{P}\!\big[\mathrm{FM}(\mathbb{Z},\pi,\eta)\ \mathrm{survives}\big]=0$.
Conversely, if
\[
\liminf_{n\to\infty} n\,\mathbb{P}(D^{\rightarrow}\ge n)\ >\ \frac{1}{E(\eta)}
\quad \textup{(with $1/E(\eta):=0$ when $E(\eta)=\infty$)},
\]
then $\mathbb{P}\!\big[\mathrm{FM}(\mathbb{Z},\pi,\eta)\ \mathrm{survives}\big]>0$.

By \eqref{eq:tail-right} and integrating over $\pi$, the decisive quantity becomes
\begin{equation}\label{eq:key-integral}
n\,\mathbb{P}(D^{\rightarrow}\ge n)\ =\ \int_{0}^{1} n\Big(\frac{1-\sqrt{1-x^{2}}}{x}\Big)^{\!n} f_\pi(x)\,dx,
\end{equation}
whenever $\pi$ has density $f_\pi$; for upper bounds we also use
\begin{equation}\label{I_ni}
I_n:=\int_{0}^{1} n\Big(\frac{1-\sqrt{1-x}}{x}\Big)^{\!n} f_\pi(x)\,dx,
\end{equation}
exploiting $\sqrt{1-x^{2}}\ge \sqrt{1-x}$, and for lower bounds
\begin{equation}\label{J_ni}
J_n:=\int_{0}^{1} n\Big(\frac{1-\sqrt{1-x^{2}}}{x}\Big)^{\!n} f_\pi(x)\,dx.
\end{equation}

So, to conclude Theorem~\ref{teomain} (i), it suffices to show that
\begin{align}
\beta>\tfrac12&:\quad I_n\xrightarrow[n\to\infty]{}0,\quad\text{hence } \limsup_{n\to\infty} n\,\mathbb{P}(D^{*}\ge n)= 0;\label{eq:beta-gt}\\[1pt]
\beta=\tfrac12&:\quad \limsup_{n\to\infty} I_n\ \le\ 2\,\limsup_{n\to\infty} L(n^{2}), \label{eq:beta-eq0}\\[1pt]
\phantom{\beta=\tfrac12}&\qquad\text{thus } \limsup_{n\to\infty} n\,\mathbb{P}(D^{*}\ge n)\le4\,\limsup_{n\to\infty} L(n^{2}).\notag
\end{align}

Likewise, we prove Theorem~\ref{teomain} (ii) by showing that
\begin{align}
\beta<\tfrac12&:\quad \liminf_{n\to\infty} J_n=\infty,\quad\text{hence } \liminf_{n\to\infty} n\,\mathbb{P}(D^{\rightarrow}\ge n)=\infty;\label{eq:beta-lt}\\[1pt]
\beta=\tfrac12&:\quad \liminf_{n\to\infty} J_n\ \ge\ \sqrt{2}\,\liminf_{n\to\infty} L(n^{2}),\label{eq:beta-eq}\\[1pt]
\phantom{\beta=\tfrac12}&\qquad \text{thus } \liminf_{n\to\infty} n\,\mathbb{P}(D^{\rightarrow}\ge n)\ge\sqrt{2}\,\liminf_{n\to\infty} L(n^{2}).\notag
\end{align}

\begin{proof}[Proof of Theorem \ref{teomain}]

Assume that for some \(\beta>0\),
\begin{equation}\label{eq:U-edge}
f_\pi(u)\ \sim\ (1-u)^{\beta-1}L\Big(\frac{1}{1-u}\Big)\qquad (u\uparrow 1).
\end{equation}

Then, for every \(\varepsilon\in(0,1)\) there exists \(\eta_\varepsilon\in(0,1)\) such that for all $u\in(1-\eta_\varepsilon,1)$,
\begin{equation}\label{limitU1}
(1-\varepsilon)\,(1-u)^{\beta-1}L\Big(\frac{1}{1-u}\Big)\ \le\ f_\pi(u)\ \le\ (1+\varepsilon)\,(1-u)^{\beta-1}L\Big(\frac{1}{1-u}\Big).
\end{equation}

\noindent
\textbf{Proof of Theorem \ref{teomain} (i):} For \(n\in\mathbb N\), define
\[
I_n\ :=\ \int_{0}^{1} n\Big(\frac{1-\sqrt{1-x}}{x}\Big)^{\!n} f_\pi(x)\,dx.
\]

Note that
\[
\frac{1-\sqrt{1-x}}{x}=\frac{1}{1+\sqrt{1-x}},
\]
hence, defining
\[
K_n(x)\ :=\ n\Big(1+\sqrt{1-x}\Big)^{-n},
\]
we have \(I=\int_{0}^{1} K_n(x)\,f_\pi(x)\,dx\).

Fix \(\varepsilon\in(0,1)\) and let \(\eta_\varepsilon\in(0,1)\) as in (\ref{limitU1}).
Split
\begin{equation}\label{An1}
I_n=A_n+B_n,\qquad
A_n:=\int_{0}^{1-\eta_\varepsilon}K_n(x)f_\pi(x)\,dx,\quad B_n:=\int_{1-\eta_\varepsilon}^{1}K_n(x)f_\pi(x)\,dx.
\end{equation}

For \(x\in(0,1-\eta_\varepsilon]\), \(\sqrt{1-x}\ge \sqrt{\eta_\varepsilon}\), hence
\begin{equation}\label{An2}
0\le K_n(x)\ \le\ n\,\rho^n,\qquad \rho:=\frac{1}{1+\sqrt{\eta_\varepsilon}}\in(0,1).
\end{equation}

Therefore, \(0\le A_n\le n\,\rho^n\), which tends to \(0\).

Now, for \(x\in (1-\eta_\varepsilon,1)\), set \(s=\sqrt{1-x}\) (i.e., \(x=1-s^2\), \(dx=-2s\,ds\)), then let \(t=ns\) (so \(dt=n\,ds\)).
Without any approximation,
\[
\begin{aligned}
B_n
&=\int_{0}^{\sqrt{\eta_\varepsilon}} 2n\,s\,(1+s)^{-n}\,f_\pi(1-s^2)\,ds\\
&=\frac{2}{n}\int_{0}^{n\sqrt{\eta_\varepsilon}} t\Big(1+\frac{t}{n}\Big)^{-n}\,f_\pi\!\Big(1-\frac{t^2}{n^2}\Big)\,dt.
\end{aligned}
\]

We use the standard inequality, valid for all \(t\ge 0\) and \(n\in\mathbb N\):
\[
\Big(1+\frac{t}{n}\Big)^{-n}\ \le\ e^{-t},
\]
which follows from \((1+x)^{-1}\le e^{-x}\) for \(x\ge 0\).

Moreover, for \(t\in[0,n\sqrt{\eta_\varepsilon}]\) we have \(1-\frac{t^2}{n^2}\in(1-\eta_\varepsilon,1)\), so by (\ref{limitU1}),
\[
f_\pi\!\Big(1-\frac{t^2}{n^2}\Big)\ \le\ (1+\varepsilon)\,
\Big(\frac{t^2}{n^2}\Big)^{\beta-1}L\Big(\frac{n^2}{t^2}\Big).
\]
Let \(\varepsilon\in(0,\beta)\) and \(A>1\) be fixed. Without loss of generality, assume \(\sqrt{\eta_\varepsilon}\le 1/\sqrt{X(A,\varepsilon)}\), where \(X(A,\varepsilon)\) is as in Lemma~\ref{lem:Potter}; otherwise, replace \(\eta_\varepsilon\) by \(\hat{\eta}_\varepsilon\in(0,1/X(A,\varepsilon))\) and, for simplicity of notation, continue to denote it by \(\eta_\varepsilon\). In particular, we then have \([0,n\sqrt{\eta_\varepsilon}]\subset [0,n/\sqrt{X(A,\varepsilon)}]\). Combining the previous displays and taking \(n\) sufficiently large, by \eqref{eq:Potter-global} from Lemma~\ref{lem:Potter}, we obtain

\begin{equation}\label{Bn1}
\begin{aligned}
0\le B_n
&\le \frac{2}{n}\int_{0}^{n\sqrt{\eta_\varepsilon}} t\Big(1+\frac{t}{n}\Big)^{-n}
\,(1+\varepsilon)\Big(\frac{t^2}{n^2}\Big)^{\beta-1}L\!\Big(\frac{n^2}{t^2}\Big)\,dt\\[3pt]
&\le 2(1+\varepsilon)\,n^{1-2\beta}\int_{0}^{n\sqrt{\eta_\varepsilon}}
t^{2\beta-1}\Big(1+\frac{t}{n}\Big)^{-n}A\,\max\{t^{2\varepsilon},\,t^{-2\varepsilon}\}\,L(n^{2})\,dt\\[3pt]
&\le 2(1+\varepsilon)A\,n^{1-2\beta}L(n^2)\Bigg(\int_{1}^{\infty} t^{2(\beta+\varepsilon)-1}e^{-t}\,dt+\int_{0}^{1} t^{2(\beta-\varepsilon)-1}e^{-t}\,dt\Bigg)\\[3pt]
&= 2(1+\varepsilon)A_{\varepsilon,\beta}\,\,n^{1-2\beta}L(n^2).
\end{aligned}
\end{equation}
where \(A_{\varepsilon,\beta}:=A\Big(\int_{1}^{\infty} t^{2(\beta+\varepsilon)-1}e^{-t}\,dt+\int_{0}^{1} t^{2(\beta-\varepsilon)-1}e^{-t}\,dt\Big)<\infty\). 

If \(\beta>\frac{1}{2}\) then \(1-2\beta<0\) and \(n^{\,1-2\beta}L(n^2)\to 0\) as \(n\to\infty\). Therefore, by (\ref{An1})–(\ref{Bn1}),
\begin{equation}\label{epsilonc}
0\ \le\ I_n\ \le\ n\,\rho^n\ +\ 2(1+\varepsilon)A_{\varepsilon,\beta}\,\,n^{1-2\beta}L(n^2).
\end{equation}
Also \(n\rho^n\to 0\). Since \(\varepsilon\in(0,\beta)\) is arbitrary, it follows that
\[
\lim_{n\to\infty} I_n=0.
\]

In addition, if $\beta=\tfrac{1}{2}$, by the Dominated Convergence Theorem, letting $\varepsilon\downarrow 0$ and $A\downarrow 1$ in \eqref{epsilonc}, we obtain
\[
\limsup_{n\to\infty} I_n\ \le\ 2\,\limsup_{n\to\infty}L(n^2).
\]

Thus, \eqref{eq:beta-gt} and \eqref{eq:beta-eq0} hold, and item \textup{(i)} follows.

\(\hfill\square\)




\noindent
\textbf{Proof of Theorem \ref{teomain} (ii):}


For \(n\in\mathbb N\), define \(R_n(x):=n\big(\frac{1-\sqrt{1-x^2}}{x}\big)^{n}\), and
\[
J_n\ :=\ \int_{0}^{1} R_n(x)\,f_\pi(x)\,dx.
\]



For \(x\in (0,1)\), set \(s=\sqrt{1-x^2}\) (so \(x=\sqrt{1-s^2}\) and \(dx=-\frac{s}{\sqrt{1-s^2}}\,ds\)).
Then
\[
\begin{aligned}
J_n
&=\int_{0}^{1} n\Big(\frac{1-s}{\sqrt{1-s^2}}\Big)^{n}
\frac{s}{\sqrt{1-s^2}}\,f_\pi\big(\sqrt{1-s^2}\big)\,ds.
\end{aligned}
\]
Making the scaling \(t=ns\) (so \(ds=dt/n\) and \(t\in(0,n)\)) leads to
\begin{equation}\label{eq:Bn-t}
J_n=\frac{1}{n}\int_{0}^{n} t\,\Biggl(\frac{1-\frac{t}{n}}{\sqrt{\,1-\frac{t^2}{n^2}\,}}\Biggr)^{\!n}
\frac{1}{\sqrt{\,1-\frac{t^2}{n^2}\,}}\ f_\pi\!\biggl(\sqrt{\,1-\frac{t^2}{n^2}\,}\biggr)\,dt.
\end{equation}

We now seek suitable bounds for each factor in the integrand. For \(s\in(0,1)\), define 
\[r(s):=\frac{1-s}{\sqrt{1-s^2}}\in(0,1).\] 
By Lemma \ref{lemma_aproxt} (iii) \(\lim_{s\downarrow 0}\frac{-\log r(s)}{s}=1\). By the definition of limit, for any \(\varepsilon\in(0,\beta)\subset(0,1/2]\), there exists \(\delta_{\varepsilon,1}\in(0,1)\) such that \(\bigl|\frac{-\log r(s)}{s}-1\bigr|<\varepsilon\) for all \(s\in(0,\delta_{\varepsilon,1}]\). Hence
\[
(1-\varepsilon)\,s\ \le\ -\log r(s)\ \le\ (1+\varepsilon)\,s,\qquad s\in(0, \delta_{\varepsilon,1}].
\]
Exponentiating yields \(e^{-(1+\varepsilon) s} \le r(s)\le e^{-(1-\varepsilon) s}\) for all \(s\in[0,\delta_{\varepsilon,1}]\). Taking again \(s=t/n\), we deduce
\begin{equation}\label{cotar}
e^{-(1+\varepsilon) t}\ \le\ \Biggl(\frac{1-\frac{t}{n}}{\sqrt{1-\frac{t^2}{n^2}}}\Biggr)^{\!n}
=r\!\Bigl(\frac{t}{n}\Bigr)^{\!n}\ \le\ e^{-(1-\varepsilon) t},\qquad t\in(0, n\delta_{\varepsilon,1}].
\end{equation}

Moreover, for any fixed \(t>0\), the bounds in \eqref{cotar} can be made arbitrarily close to \(e^{-t}\), and choosing $n$ large enough so that \(t<n\delta_{\varepsilon,1}\) ensures that the inequality holds. Consequently, we have the pointwise convergence
\begin{equation}\label{eq:converge}
    \lim_{n\to \infty}r\!\Bigl(\frac{t}{n}\Bigr)^{\!n}=e^{-t}.
\end{equation}




Next, we bound \(f_\pi(\sqrt{1-s^2})\) for $s$ near \(0\). Note that by Lemma \ref{lemma_aproxt} (ii), we have, as \(s\downarrow0\),
\[
1-\sqrt{1-s^2}= \frac{s^2}{2}+\mathcal O(s^4)=\frac{s^2}{2}(1+\mathcal O(s^2)).
\]
Applying Lemma \ref{lemma_aproxt} (i) gives \((1-\sqrt{1-s^2})^{\beta-1}=2^{\,1-\beta} s^{2\beta-2}(1+\mathcal O(s^2))\) as \(s\downarrow0\). Therefore, for every \(\varepsilon\in(0,\beta)\), there exists \(\delta_{\varepsilon,2}\in(0,\delta_{\varepsilon,1}]\) such that
\[(1-\varepsilon)\,2^{\,1-\beta}\,s^{2\beta-2}\ \le\ \bigl(1-\sqrt{1-s^2}\bigr)^{\beta-1},
\qquad s\in(0,\delta_{\varepsilon,2}].\]

Combining this with equations (\ref{epsilon_L_2}) and (\ref{limitU1}) in which \(u=\sqrt{1-s^2}\), we conclude that there exists \(\delta^*_{\varepsilon}\in(0,\delta_{\varepsilon,2}]\) such that
\begin{equation}\label{eq:fU-two-sided}
(1-\varepsilon)^3\,2^{\,1-\beta}\,s^{2\beta-2}L\Big(\frac{1}{s^2}\Big)\ \le\
f_\pi\big(\sqrt{1-s^2}\big),
\qquad s\in(0,\delta^*_{\varepsilon}].
\end{equation}

Now define 
\begin{equation}\label{D_n}
D_n:=\frac{1}{n}\int_{0}^{n\delta^*_{\varepsilon}}t\,\Biggl(\frac{1-\frac{t}{n}}{\sqrt{\,1-\frac{t^2}{n^2}\,}}\Biggr)^{\!n}
\frac{1}{\sqrt{\,1-\frac{t^2}{n^2}\,}}\ f_\pi\!\biggl(\sqrt{\,1-\frac{t^2}{n^2}\,}\biggr)\,dt\
\end{equation}
and note that \(D_n\le J_n\).

Combining \eqref{D_n} with \eqref{eq:fU-two-sided} and applying the Potter's bound \eqref{eq:Potter-global} from Lemma~\ref{L_equiv}, we obtain

\[
\begin{aligned}
D_{n}&\ \ge\ n^{1-2\beta}\int_{0}^{T} t\,\Biggl(\frac{1-\frac{t}{n}}{\sqrt{\,1-\frac{t^2}{n^2}\,}}\Biggr)^{\!n}
\frac{1}{\sqrt{\,1-\frac{t^2}{n^2}\,}}\ (1-\varepsilon)^2\,2^{\,1-\beta}
t^{2\beta-2}L\Bigl(\frac{n^2}{t^2}\Bigr)\,dt\\[3pt]
&\ge (1-\varepsilon)^2A^{-1}2^{\,1-\beta}\,\ n^{1-2\beta}L(n^2)\int_{0}^{T} t\,\Biggl(\frac{1-\frac{t}{n}}{\sqrt{\,1-\frac{t^2}{n^2}\,}}\Biggr)^{\!n}
\frac{1}{\sqrt{\,1-\frac{t^2}{n^2}\,}}\,
t^{2\beta-2}\,\min\{t^{2\varepsilon},\,t^{-2\varepsilon}\}dt
\end{aligned}
\]

for any fixed \(T>0\), \(\varepsilon\in(0,\beta)\), \(A>1\), and sufficiently large \(n\in\mathbb{N}\) so that $T\le n\delta^*_\varepsilon$. 

Additionally, for \(t\in[0,T]\),
\[
1\leq \frac{1}{\sqrt{\,1-\frac{t^2}{n^2}\,}}\leq \frac{1}{\sqrt{\,1-\frac{T^2}{n^2}\,}}\ \xrightarrow[n\to\infty]{}\ 1,
\]
and by (\ref{cotar}) there exists \(N\in \mathbb{N}\) such that, if \(n\geq N\),
\[
t\,\Biggl(\frac{1-\frac{t}{n}}{\sqrt{\,1-\frac{t^2}{n^2}\,}}\Biggr)^{\!n}
\frac{1}{\sqrt{\,1-\frac{t^2}{n^2}\,}}\ 
t^{2\beta-2}\min\{t^{2\varepsilon},\,t^{-2\varepsilon}\}
\ \leq\ 2 e^{-(1-\varepsilon) t}t^{2\beta-1}\min\{t^{2\varepsilon},\,t^{-2\varepsilon}\},\qquad t\in [0,T].
\]

Moreover, \(2 e^{-(1-\varepsilon) t}t^{2\beta-1}\min\{t^{2\varepsilon},\,t^{-2\varepsilon}\}\) is integrable on \([0,T]\), and by (\ref{eq:converge}) the hypotheses of the Dominated Convergence Theorem are satisfied. Therefore,
\[
\int_{0}^{T}t\,\Biggl(\frac{1-\frac{t}{n}}{\sqrt{\,1-\frac{t^2}{n^2}\,}}\Biggr)^{\!n}
\frac{1}{\sqrt{\,1-\frac{t^2}{n^2}\,}}\ 
t^{2\beta-2}\,\min\{t^{2\varepsilon},\,t^{-2\varepsilon}\}dt\ \xrightarrow[n\to\infty]{}\ \int_{0}^{T}e^{- t}t^{2\beta-1}\, \min\{t^{2\varepsilon},\,t^{-2\varepsilon}\}\,dt.
\]

Then,
\begin{equation}\label{D_n_limit}
\liminf_{n\to\infty}\ \frac{D_{n}}{n^{\,1-2\beta}L(n^2)}
\ \ge\ (1-\varepsilon)^2A^{-1}\,2^{\,1-\beta}\int_{0}^{T} t^{2\beta-1}e^{-t}\,\min\{t^{2\varepsilon},\,t^{-2\varepsilon}\}\,dt.
\end{equation}

Note that (\ref{D_n_limit}) holds for every \(T>0\), \(\varepsilon\in(0,\beta)\), and \(A>1\). Additionally, \(\min\{t^{2\varepsilon},\,t^{-2\varepsilon}\}\le t^{-2\varepsilon}\) and \(t^{2(\beta-\varepsilon)-1}e^{-t}\) is integrable on \([0,T]\). Then, by the Dominated Convergence Theorem, letting \(\varepsilon\downarrow 0\) we obtain that 
\[
\liminf_{n\to\infty}\ \frac{D_{n}}{n^{\,1-2\beta}L(n^2)}
\ \ge\ A^{-1}\,2^{\,1-\beta}\int_{0}^{T} t^{2\beta-1}e^{-t}\,dt.
\]

Finally, letting \(T\to\infty\) and \(A\downarrow 1\) yields
\[
\liminf_{n\to\infty}\ \frac{D_{n}}{n^{\,1-2\beta}L(n^{2})}
\ \ge\ 2^{1-\beta}\Gamma(2\beta).
\]

Therefore,

\begin{equation}\label{D_n_+_lim}
\begin{cases}
\liminf_{n\to\infty} D_{n}=\infty, & \text{if }\beta<\tfrac{1}{2},\\[4pt]
\liminf_{n\to\infty} D_{n}\geq \sqrt{2}\,\liminf_{n\to\infty}L(n^2), & \text{if }\beta=\tfrac{1}{2}.
\end{cases}
\end{equation}

Since $J_n\geq D_n$, \eqref{eq:beta-lt} and \eqref{eq:beta-eq} hold, and item \textup{(ii)} follows.

\end{proof}

\subsection{Results regarding odds}

This section develops auxiliary results concerning the odds variable which provides an alternative parametrization of the survival law. 

The proofs rely on change-of-variables arguments and on uniform convergence theorems for regularly varying functions. These results ensure that heavy-tailed behavior is stable under natural transformations and confirm that the exponent $\beta$ continues to govern the asymptotics relevant to the frog model.

\begin{proof}[Proof of Proposition~\ref{proposition:cmp}]
Let \(\pi:=\frac{X}{1+X}\). By the change-of-variables theorem, \(\pi\) admits a density given by \(f_\pi(u):=f_X\!\big(u/(1-u)\big)/(1-u)^2\). For \(t=\frac{u}{1-u}\) we have \((1-u)=1/(1+t)\) and
\begin{equation}\label{limw2}
\frac{f_\pi(u)}{(1-u)^{\beta-1}\,L_X\!\big(\tfrac{1}{1-u}\big)}
  = \frac{f_X(t)\,(t+1)^{\beta+1}}{L_X(t+1)}
  = \bigg(\frac{f_X(t)\, t^{\beta+1}}{L_X(t)}\bigg)\,
    \frac{L_X(t)}{L_X(t+1)}\,
    \big(1+1/t\big)^{\beta+1}.
\end{equation}

Since $L_X$ is slowly varying at infinity, we have $L_X(t)\sim L_X(t+1)$, and therefore, by (\ref{limw}), the expression in (\ref{limw2}) converges to $1$ as $t\uparrow \infty$. This completes the proof, since $u\uparrow 1$ as $t\uparrow \infty$.

\end{proof}

\begin{proof}[Proof of Proposition~\ref{thm:B}]
Let \(Z:=XY\). By the change–of–variables theorem, consider
\[
T:\mathbb{R}\times(0,\infty)\to\mathbb{R}\times(0,\infty),\qquad
T(x,y)=(z,y)\ \text{ with }\ z=xy .
\]
For \(y>0\) the inverse map is \(T^{-1}(z,y)=(z/y,y)\), and the Jacobian determinant is
\[
\det \tfrac{\partial(x,y)}{\partial(z,y)}=\frac{1}{y}.
\]
If \(X\) and \(Y\) are independent with densities \(f_X\) and \(f_Y\), then the joint density of \((Z,Y)\) is
\[
f_{Z,Y}(z,y)=\frac{1}{y}\,f_X\!\Bigl(\frac{z}{y}\Bigr)\,f_Y(y),\qquad y>0.
\]
Integrating out \(y>0\) yields the density of \(Z\):
\[
f_Z(z)=\int_{0}^{\infty}\frac{1}{y}\,f_X\!\Bigl(\frac{z}{y}\Bigr)\,f_Y(y)\,dy,\qquad z\in\mathbb{R}.
\]

Define the rescaled integrand
\begin{equation}\label{eq:Hz-def}
H_z(y)\ :=\ \frac{1}{y}\,\frac{f_X(z/y)}{z^{-(\beta+1)}\,L_X(z)},\qquad y>0,
\end{equation}
so that
\[
\frac{f_Z(z)}{z^{-(\beta+1)}\,L_X(z)}\ =\ \int_0^\infty H_z(y)\,f_Y(y)\,dy\ =\ \mathbb{E}\!\left[H_z(Y)\right].
\]

By \eqref{eq:fX-RV}, for each fixed $y>0$,
\[
\frac{f_X(z/y)}{z^{-(\beta+1)}\,L_X(z)}
=\underbrace{\frac{f_X(z/y)}{(z/y)^{-(\beta+1)}\,L_X(z/y)}}_{\longrightarrow\ 1}
\cdot y^{\beta+1}\cdot
\underbrace{\frac{L_X(z/y)}{L_X(z)}}_{\longrightarrow\ 1}
\qquad (z\to\infty).
\]
Hence, by \eqref{eq:Hz-def},
\begin{equation}\label{eq:pointwise}
H_z(y)\ \xrightarrow[z\to\infty]{}y^{\beta}\qquad\text{for every }y>0.
\end{equation}

We use Potter's bounds for regularly varying functions (see \cite[Thm. 1.5.6]{BGT1987}): for every $\delta\in(0,1)$ and $C> 1$ there exist a constant $t_{\delta,C}$ such that
\begin{equation}\label{eq:potter-RV}
\frac{f_X(tu)}{f_X(t)}\ \le C\,u^{-(\beta+1)+\delta}, \, u\ge 1, \qquad t\ge t_{\delta,C}.
\end{equation}

Moreover, by \eqref{eq:fX-RV}, there exists $t_1$ such that $f_X(t)\le 2\,t^{-(\beta+1)}L_X(t)$ for all $t\ge t_1$.
Fix $\delta\in(0,\varepsilon)$ and take $z\ge z_0:=\max\{t_{\delta,C},t_1\}$. With $t=z$ and $u=1/y$ we obtain
\begin{equation}\label{ct1}
\frac{f_X(z/y)}{z^{-(\beta+1)}\,L_X(z)}
=\frac{f_X(z/y)}{f_X(z)}\cdot\frac{f_X(z)}{z^{-(\beta+1)}L_X(z)}
\ \le\ 2\,C\, y^{\beta+1+\delta}, \,\, 0<y\le 1,\\
\end{equation}
Moreover, let \(g(z):=f(z)z^{\beta+1+\delta}\); it is easy to show that \(g\) is regularly varying with index \(\delta\) and the Uniform Convergence Theorem
(see, e.g., \cite[Thm.~1.5.2]{BGT1987}) yields that for every \(b>0\)
\[
\sup_{\lambda\in (0,b]}\Bigg|\frac{g(\lambda z)}{g(z)}-\lambda^{\delta}\Bigg|\xrightarrow[z\to\infty]{}0.
\]

In particular, taking \(b=1\). Now, given \(\varepsilon>0\), there exists \(z_1:=z_1(\varepsilon)>0\) such that if \(z\geq z_1\), then \(\frac{g(\lambda z)}{g(z)\lambda^{\delta}}\leq \varepsilon\lambda^{-\delta}+1\) for all \(\lambda \in (0,1]\). Then, for \(z\geq \max\{z_1,t_1\}\) and \(y>1\), that is, \(\frac{1}{y}\in (0,1]\), we have
\begin{equation}\label{ct2}
\frac{f_X(z/y)}{z^{-(\beta+1)}\,L_X(z)}
=\frac{f_X(z/y)\frac{1}{y^{\beta+1+\delta}}}{f_X(z)\frac{1}{y^{\delta}}}\cdot\frac{1}{\frac{1}{y^{\beta+1}}}\cdot\frac{f_X(z)}{z^{-(\beta+1)}L_X(z)}
\ \le\ 2\, y^{\beta+1}\,(\varepsilon y^{\delta}+1), \,\, y> 1,\\
\end{equation}

Dividing by $y$ (cf. \eqref{eq:Hz-def}) yields the \emph{uniform} bound
\begin{equation}\label{eq:uniform-dom}
0\ \le\ H_z(y)\ \le\ 2\,C y^{\beta+\delta}\,\mathbf{1}_{(0,1]}(y)+2\, y^{\beta}\,(\varepsilon y^{\delta}+1)\,\mathbf{1}_{(1,\infty)}(y),
\quad z\ge z_0\vee z_1.
\end{equation}

Let \(G(y):=2\,C y^{\beta+\delta}\,\mathbf{1}_{(0,1]}(y)+2\, y^{\beta}\,(\varepsilon y^{\delta}+1)\,\mathbf{1}_{(1,\infty)}(y)\). Taking expectation with respect to \(Y\), we have that \(\mathbb{E}\!\left[G(Y)\right]<\infty\). Since \(\delta<\varepsilon\) and \(\mathbb{E}[Y^{\beta+\varepsilon}]<\infty\), we have \(\mathbb{E}[Y^{\beta+\delta}]<\infty\). Therefore, there exists an integrable random variable \(G(Y)\) such that \(0\le H_z(Y)\le G(Y)\) for all \(z\ge z_0\vee z_1\). Then, the Dominated Convergence Theorem gives
\[
\frac{f_Z(z)}{z^{-(\beta+1)}\,L_X(z)}
=\mathbb{E}[H_z(Y)]\ \xrightarrow[z\to\infty]{}\ \mathbb{E}[\,Y^{\beta}].
\]
Multiplying both sides by \(z^{-(\beta+1)}L_X(z)\) yields the asserted asymptotic \eqref{eq:fZ-asymp}.

\end{proof}

\section{Examples}\label{Examples}

This chapter presents examples that illustrate the scope of the main theorem and clarify how the parameter $\beta$ and the slowly varying function $L$ arise in concrete settings. These cases demonstrate that the exponent $\beta$ alone determines the phase transition, while $L$ influences only the boundary behavior. The examples also show that a wide range of probabilistic models naturally fit into the framework developed in the previous sections.

\subsection{Distributions satisfying the main condition}\label{subs:ex1}

This section lists representative survival-parameter distributions whose densities satisfy condition~(\ref{eq:edge-behavior}).

The examples include elementary choices, such as power-law densities, as well as distributions obtained through transformations of Gamma variables. These explicit computations illustrate how diverse families of distributions fall within the scope of the general asymptotic condition used in Theorem~2.2.

\begin{example} Let $\beta>0$ and define, for $u\in(0,1)$,
\[
f_\pi(u)=\frac{(1-u)^{\beta-1}}{\beta}.
\]
Then, $f_\pi$ satisfies condition~(\ref{eq:edge-behavior}), with the explicit choices
\[
\beta\, \text{ as given},\qquad L((1-u)^{-1})\equiv \frac{1}{\beta}.
\]
\end{example}

\begin{example}\label{example2}
Let $X\sim\mathrm{Gamma}(\alpha,\lambda_1)$ and $Y\sim\mathrm{Gamma}(\beta,\lambda_2)$ be independent, with shapes $\alpha,\beta>0$
and (rate) parameters $\lambda_1,\lambda_2>0$. Define
\[
\pi:=\frac{X}{X+Y}\in(0,1).
\]
We shall prove that $\pi$ has density, for $0<u<1$,
\[
f_\pi(u)
=\frac{\Gamma(\alpha+\beta)}{\Gamma(\alpha)\Gamma(\beta)}\,
\frac{\lambda_1^{\alpha}\lambda_2^{\beta}\,u^{\alpha-1}(1-u)^{\beta-1}}
{\big(\lambda_2+(\lambda_1-\lambda_2)u\big)^{\alpha+\beta}}.
\]

Therefore,
\[
f_\pi(u)\sim \frac{\Gamma(\alpha+\beta)\,\lambda_2^{\beta}}{\Gamma(\alpha)\Gamma(\beta)\,\lambda_1^{\beta}}\,
(1-u)^{\beta-1}\qquad (u\uparrow 1).
\]
Thus, $f_\pi$ fits condition \eqref{eq:edge-behavior}, with the explicit choices
\[
\beta\, \text{ as given},\qquad
L((1-u)^{-1})\equiv \frac{\Gamma(\alpha+\beta)\,\lambda_2^{\beta}}{\Gamma(\alpha)\Gamma(\beta)\,\lambda_1^{\beta}}.
\]

Moreover, for $\lambda_1=\lambda_2$, this reduces to the case where $\pi\sim\mathrm{Beta}(\alpha,\beta)$. 

\begin{proof}
Let
\[
\pi=\frac{X}{X+Y}\in(0,1),\qquad V=X+Y\in(0,\infty).
\]

By independence,
\[
f_{X,Y}(x,y)=f_X(x)f_Y(y)
=\frac{\lambda_1^{\alpha}}{\Gamma(\alpha)}x^{\alpha-1}e^{-\lambda_1 x}\,
  \frac{\lambda_2^{\beta}}{\Gamma(\beta)}y^{\beta-1}e^{-\lambda_2 y}\,
  \mathbf{1}_{(0,\infty)^2}(x,y).
\]

Consider the bijection $T:(0,\infty)^2\to (0,1)\times(0,\infty)$ defined by
\[
(u,v)=T(x,y):=\Big(\tfrac{x}{x+y},\,x+y\Big),\qquad \text{so that}\qquad
(x,y)=(uv,\,(1-u)v).
\]
The support maps as $(x,y)\in(0,\infty)^2 \Longleftrightarrow (u,v)\in(0,1)\times(0,\infty)$.

The Jacobian determinant of the inverse map $(u,v)\mapsto (x,y)$ equals
\[
J(u,v)=\det
\begin{pmatrix}
\partial x/\partial u & \partial x/\partial v\\[2pt]
\partial y/\partial u & \partial y/\partial v
\end{pmatrix}
=
\det\begin{pmatrix}
v & u\\[2pt]
-\,v & 1-u
\end{pmatrix}=v.
\]

For $u\in(0,1)$ and $v>0$,
\[
\begin{aligned}\label{density}
f_{\pi,V}(u,v)
&= f_{X,Y}(uv,(1-u)v)\; |J(u,v)|\\
&=\frac{\lambda_1^{\alpha}}{\Gamma(\alpha)}(uv)^{\alpha-1}e^{-\lambda_1 uv}\,
  \frac{\lambda_2^{\beta}}{\Gamma(\beta)}\big((1-u)v\big)^{\beta-1}e^{-\lambda_2(1-u)v}
  \,|v|\\
&=\frac{\lambda_1^{\alpha}\lambda_2^{\beta}}{\Gamma(\alpha)\Gamma(\beta)}\;
  u^{\alpha-1}(1-u)^{\beta-1}\; v^{\alpha+\beta-1}\;
  e^{-\{\lambda_2+(\lambda_1-\lambda_2)u\}\,v}.
\end{aligned}
\]

Set
\[
A(u):=\lambda_2+(\lambda_1-\lambda_2)u
\quad\text{so that}\quad A(u)>0 \ \text{for}\ u\in(0,1).
\]

Integrating $f_{\pi,V}(u,v)$ in $v$,
\[
\begin{aligned}
f_\pi(u)
&=\int_0^{\infty} f_{U,V}(u,v)\,dv\\
&=\frac{\lambda_1^{\alpha}\lambda_2^{\beta}}{\Gamma(\alpha)\Gamma(\beta)}\,
  u^{\alpha-1}(1-u)^{\beta-1}
  \int_0^{\infty} v^{\alpha+\beta-1} e^{-A(u)v}\,dv.
\end{aligned}
\]
Using the gamma integral $\int_0^{\infty} v^{k-1}e^{-a v}\,dv=\Gamma(k)a^{-k}$ (valid for $a,k>0$) with $k=\alpha+\beta$ and $a=A(u)$, we obtain
\[
\begin{aligned}
f_\pi(u)
&=\frac{\lambda_1^{\alpha}\lambda_2^{\beta}}{\Gamma(\alpha)\Gamma(\beta)}\,
  u^{\alpha-1}(1-u)^{\beta-1}\;
  \Gamma(\alpha+\beta)\, \big(A(u)\big)^{-(\alpha+\beta)}\\
&=\frac{\Gamma(\alpha+\beta)}{\Gamma(\alpha)\Gamma(\beta)}\,
  \frac{\lambda_1^{\alpha}\lambda_2^{\beta}\, u^{\alpha-1}(1-u)^{\beta-1}}
       {\big(\lambda_2+(\lambda_1-\lambda_2)u\big)^{\alpha+\beta}},
\qquad 0<u<1,
\end{aligned}
\]
which is the claimed density.


\end{proof}

\end{example}

\begin{example}
The generalized Beta distribution of the first kind has density
\[
f_\pi(u;a,b,c)=\frac{c}{B(a,b)}\,u^{ac-1}\,(1-u^{c})^{b-1},\qquad 0<u<1,
\]
with $a,b,c>0$. As $u\uparrow1$, $u^{ac-1}\to1$ and $1-u^{c}\sim c(1-u)$, hence
\[
f_\pi(u)\sim \frac{c^{\,b}}{B(a,b)}\,(1-u)^{b-1}.
\]
Therefore, $f_\pi$ satisfies condition~\eqref{eq:edge-behavior}, with
\[
\beta=b,\qquad L((1-u)^{-1})\equiv \frac{c^{\,b}}{B(a,b)}.
\]
\end{example}

\subsection{Inconclusive regime}\label{subs:ex2}

Here we give an example where the conditions of Theorem~\ref{teomain} and Remark~\ref{rem1} are not satisfied, and yet we may find conclusions about survival.

Let
\[
L(x):=(\log x)^{\sin(\log\log x)}
=\exp\!\big((\log\log x)\,\sin(\log\log x)\big), 
\qquad x>1.
\]
By Lemma~\ref{lem:L_osc}, \(L\) is slowly varying at infinity. 
Consider $\pi$ such that
\[
f_{\pi}(u)\sim (1-u)^{-1/2}\,L\!\big((1-u)^{-1}\big)
\qquad\text{as }u\uparrow 1.
\]
Again by Lemma~\ref{lem:L_osc},
\[
\liminf_{x\to\infty} L(x)=0
\qquad\text{and}\qquad
\limsup_{x\to\infty} L(x)=\infty,
\]
so Theorem~\ref{teomain} and Remark~\ref{rem1} are inconclusive in this regime.

From the proof of Theorem~\ref{teomain} and inequalities \eqref{eq:beta-eq0} and \eqref{eq:beta-eq}, we obtain
\begin{equation}
\sqrt{2}
\ \le\ \liminf_{n\to\infty}\frac{n\,\mathbb P(D^{\rightarrow}\ge n)}{L(n^{2})}
\ \le\ \limsup_{n\to\infty}\frac{n\,\mathbb P(D^{\rightarrow}\ge n)}{L(n^{2})}
\ \le\ 2\ 
\end{equation}
Moreover, by Lemma~\ref{lem:G_limits},
\[
\liminf_{n\to\infty}
n\,\mathbb P(D^{\rightarrow}\ge n)\,
\log\!\Big(\frac{1}{\mathbb P(D^{\rightarrow}\ge n)}\Big)
\ \ge\ \frac{1}{\sqrt{2}}.
\]

Let \(\text{FW}(N,R)\) denote the firework process with the number of stations given by $N$ and the radii given by $R$ (see \cite{Bertacchi2013}). 
As observed in \cite{CarvalhoMachado2025}, when \(N_z=\eta_z\) and 
\(R_{z,i}=D_{z,i}^{\rightarrow}\), the process \(\text{FW}(N,R)\) is dominated by 
\(\text{FM}(\mathbb{Z},\pi,\eta)\) in the sense that
\[
\mathbb P[\text{FW}(N,R)\ \text{percolates}]
\ \le\ 
\mathbb P[\text{FM}(\mathbb{Z},\pi,\eta)\ \text{survives}].
\]

Assuming \(\mathbb P(\eta>n)\sim c\,n^{-1}\) for some \(c>0\),
Corollary~3.3 (3) of \cite{Bertacchi2013} yields survival whenever \(c>\sqrt{2}\).
  
\subsection{Odds satisfying the conditions}\label{subs:ex3}

Table~\ref{tab:rv-families} compiles examples of densities satisfying the tail condition \eqref{limw}.
Precisely, each entry provides a random variable \(X\) with
\(f_X(t)\sim t^{-(\beta+1)}L_X(t)\) as \(t\to\infty\). Proposition~\ref{proposition:cmp} implies that for \(\pi:=X/(1+X)\),
\[
f_\pi(u)\ \sim (1-u)^{\beta-1}\,L_X\!\Bigl(\frac{1}{1-u}\Bigr)\qquad (u\uparrow 1).
\]
Parameter supports are indicated in the first column; the corresponding values of \(\beta\) and the slowly varying factor \(L_X\) appear in the remaining columns.

\begin{table}[th!]
\centering
\footnotesize
\setlength{\tabcolsep}{4pt}
\renewcommand{\arraystretch}{1.05}
\caption{Examples of densities with regularly varying tails: \(f_X(t)\sim t^{-(\beta+1)}L_X(t)\) as \(t\to\infty\).}
\label{tab:rv-families}
\begin{tabularx}{\linewidth}{@{} l X c l @{}}
\toprule
\textbf{Family (parameters)} & \textbf{Representative pdf \(f_X(t)\)} & \(\boldsymbol{\beta}\) & \(\boldsymbol{L_X(t)}\) \\
\midrule

Pareto I \((x_m>0,\ \alpha>0)\) &
\( \alpha\,x_m^{\alpha}\,t^{-(\alpha+1)}\,\mathbf{1}_{\{t\ge x_m\}} \) &
\(\alpha\) &
\( \alpha\,x_m^{\alpha} \) \\

Lomax (Pareto II) \((\alpha>0,\ \lambda>0)\) &
\( \frac{\alpha}{\lambda}\,(1+t/\lambda)^{-(\alpha+1)} \) &
\(\alpha\) &
\( \alpha\,\lambda^{\alpha} \) \\

Generalized Pareto \((\xi>0,\ \sigma>0)\) &
\( \sigma^{-1}\bigl(1+\xi t/\sigma\bigr)^{-1/\xi - 1} \) &
\( 1/\xi \) &
\( \xi^{-(1/\xi+1)}\,\sigma^{1/\xi} \) \\

Beta prime \((a>0,\ b>0)\) &
\( B(a,b)^{-1}\,t^{a-1}(1+t)^{-(a+b)} \) &
\( b \) &
\( B(a,b)^{-1} \) \\

\(F(d_1>0,\ d_2>0)\) &
{\scriptsize
\(
\begin{aligned}[t]
  & B\Bigl(\tfrac{d_1}{2},\tfrac{d_2}{2}\Bigr)^{-1}
    \Bigl(\tfrac{d_1}{d_2}\Bigr)^{d_1/2}\, t^{d_1/2-1} \\
  & \times \bigl(1+\tfrac{d_1}{d_2}t\bigr)^{-(d_1+d_2)/2}
\end{aligned}
\)
} &
\( d_2/2 \) &
\( B\Bigl(\tfrac{d_1}{2},\tfrac{d_2}{2}\Bigr)^{-1}
   \Bigl(\tfrac{d_2}{d_1}\Bigr)^{d_2/2} \) \\

Burr XII  \((c>0,\ k>0)\) &
\( c\,k\,t^{c-1}(1+t^{c})^{-(k+1)} \) &
\( c k \) &
\( c k \) \\

Inverse--Gamma \((a>0,\ b>0)\) &
\( b^{a}\Gamma(a)^{-1}\,t^{-(a+1)}e^{-b/t} \) &
\( a \) &
\( b^{a}\Gamma(a)^{-1} e^{-b/t} \) \\

Fréchet \((\alpha>0)\) &
\( \alpha\,t^{-(\alpha+1)}e^{-t^{-\alpha}} \) &
\( \alpha \) &
\( \alpha\,e^{-t^{-\alpha}} \) \\

Symmetric \(\mu\)-stable \((\mu\in(0,2))\) &
(no closed form; known tail) &
\( \mu \) &
\( \dfrac{\mu\,\Gamma(\mu)\sin(\pi\mu/2)}{\pi} \) \\

Log--Pareto \((\alpha>0,\ \rho\in\mathbb{R})\) &
\( C_{\alpha,\rho}\,t^{-(\alpha+1)}(\log t)^{\rho}\,\mathbf{1}_{\{t\ge e\}} \) &
\( \alpha \) &
\( C_{\alpha,\rho} (\log t)^{\rho} \) \\
\bottomrule
\end{tabularx}

\vspace{0.3em}
\emph{Notes.} The indicator \(\mathbf{1}_{\{\cdot\}}\) is shown only when the support differs from \(t>0\) (Pareto I, Log--Pareto).
\(B(\cdot,\cdot)\) is the Beta function; \(\Gamma(\cdot)\) the Gamma function; \(\Gamma(s,x)\) the upper incomplete gamma function.
For Log--Pareto, \( C_{\alpha,\rho} = \alpha^{\rho+1}/\Gamma(\rho+1,\alpha) \).
\end{table}

\medskip
\noindent

\section{Auxiliary results}
\label{appendix}

Appendix~\ref{appendix} collects several analytic estimates used throughout the proofs 
in Section~3. These include Taylor-type expansions, square-root and 
logarithmic approximations, and bounds for slowly varying functions such 
as Potter estimates and uniform convergence results.

\begin{lemma}\label{lemma_aproxt} (i)
Let $\gamma\in\mathbb{R}$ and define $\varphi(x)=(1+x)^{\gamma}$ for $x>-1$. 
Fix any $\delta\in(0,1)$. Then, there exists a constant $C=C(\gamma,\delta)>0$ such that
\[
\bigl|(1+x)^{\gamma}-\bigl(1+\gamma x\bigr)\bigr|\ \le\ C\,x^{2}\qquad\text{for all }|x|\le \delta.
\]
Equivalently,
\[
(1+x)^{\gamma}\ =\ 1+\gamma x + \mathcal{O}(x^{2})\qquad (x\to 0).
\]

If $\theta=\theta(s)\to 0$ as $s\to 0$, then
\[
(1+\theta(s))^{\gamma}=1+\gamma\,\theta(s)+\mathcal{O}\!\big(\theta(s)^{2}\big).
\]
In particular, if $\theta(s)=\mathcal{O}(s^{k})$ for some $k>0$, then
\[
(1+\theta(s))^{\gamma}=1+\gamma\,\theta(s)+\mathcal{O}(s^{2k}).
\]

(ii) There exist $\delta\in(0,1)$ and $C>0$ such that, for all $s$ with $|s|\le \delta$,
\[
\Bigl|\,\bigl(1-\sqrt{1-s^2}\bigr)-\tfrac12\,s^2\,\Bigr|\ \le\ C\,s^4.
\]
Equivalently,
\[
1-\sqrt{1-s^2}\ =\ \tfrac12\,s^2+\mathcal O(s^4)\qquad (s\to 0).
\]

(iii) Let \(r(s):=\dfrac{1-s}{\sqrt{1-s^2}}\) for \(s\in(-1,1)\). 
There exist \(\delta\in(0,1)\) and \(K>0\) such that, for all \(|s|\le \delta\),
\[
\bigl|\log r(s)+s\bigr|\ \le\ K\,|s|^{3}.
\]
Equivalently, \(\log r(s)=-s+\mathcal O(s^{3})\) as \(s\to 0\).
\end{lemma}

\begin{proof} Proof of (i): Let  $x\in [-\delta,\delta]$ we have $1+x\ge 1-\delta>0$, hence $\varphi$ is $C^{2}$ there with
\[
\varphi'(x)=\gamma(1+x)^{\gamma-1},\qquad 
\varphi''(x)=\gamma(\gamma-1)(1+x)^{\gamma-2}.
\]
By Taylor’s theorem with Lagrange remainder at $0$,
\[
(1+x)^{\gamma}=\varphi(0)+\varphi'(0)\,x+\frac{\varphi''(\xi_x)}{2}\,x^{2}
=1+\gamma x+\frac{\gamma(\gamma-1)}{2}\,(1+\xi_x)^{\gamma-2}\,x^{2},
\]
for some $\xi_x$ between $0$ and $x$. Since $(1+\xi)^{\gamma-2}$ is continuous on $[-\delta,\delta]$,
\[
M:=\sup_{|\xi|\le \delta}\bigl|(1+\xi)^{\gamma-2}\bigr|<\infty,
\]
and hence
\[
\bigl|(1+x)^{\gamma}-(1+\gamma x)\bigr|
\le \frac{|\gamma(\gamma-1)|}{2}M\,x^{2}
=:C(\gamma,\delta)\,x^{2}.
\]
Proof of (ii): Consider $\varphi(x):=\sqrt{1-x}$ on $x\in[0,1)$. Then
\[
\varphi(0)=1,\qquad \varphi'(x)=-\tfrac12(1-x)^{-1/2},\qquad 
\varphi''(x)=-\tfrac14(1-x)^{-3/2}.
\]
By Taylor's theorem with Lagrange remainder at $0$, for any $x\in[0,\delta^2]$ we have
\[
\varphi(x)=\varphi(0)+\varphi'(0)\,x+\frac{\varphi''(\xi_x)}{2}\,x^2
= 1-\tfrac12 x+\frac{-\tfrac14(1-\xi_x)^{-3/2}}{2}\,x^2
= 1-\tfrac12 x-\tfrac18(1-\xi_x)^{-3/2}x^2,
\]
for some $\xi_x\in(0,x)$. Rearranging,
\[
1-\sqrt{1-x}=\tfrac12\,x+\tfrac18(1-\xi_x)^{-3/2}\,x^2.
\]
Fix $\delta\in(0,1)$; then for all $x\in[0,\delta^2]$ we have
$(1-\xi_x)^{-3/2}\le (1-\delta^2)^{-3/2}$. Hence,
\[
\Bigl|\,\bigl(1-\sqrt{1-x}\bigr)-\tfrac12 x\,\Bigr|
\ \le\ \frac{1}{8}(1-\delta^2)^{-3/2}\,x^2
\ =: C_0\,x^2.
\]
Finally, set $x=s^2$ to obtain, for $|s|\le \delta$,
\[
\Bigl|\,\bigl(1-\sqrt{1-s^2}\bigr)-\tfrac12\,s^2\,\Bigr|\ \le\ C_0\,s^4,
\]
which proves the claim with $C=C_0$.

Proof of (iii): Write \(\log r(s)=\phi(s)+\psi(s)\) with
\[
\phi(s):=\log(1-s),\qquad \psi(s):=-\tfrac12\log(1-s^{2}).
\]

On \([-\delta,\delta]\subset(-1,1)\), \(\phi\) is \(C^{3}\) with
\[
\phi(0)=0,\quad \phi'(s)=-(1-s)^{-1},\quad \phi''(s)=-(1-s)^{-2},\quad \phi'''(s)=-2(1-s)^{-3}.
\]
Taylor’s theorem with Lagrange remainder at \(0\) yields
\[
\phi(s)=\phi(0)+\phi'(0)s+\tfrac12\phi''(0)s^{2}+\tfrac16\phi'''(\xi_s)s^{3}
=-s-\tfrac12 s^{2}+R_\phi(s),
\]
for some \(\xi_s\) between \(0\) and \(s\), with
\[
|R_\phi(s)|=\frac{|\phi'''(\xi_s)|}{6}\,|s|^{3}
\le \frac{2}{6}(1-\delta)^{-3}\,|s|^{3}
=:C_1\,|s|^{3}.
\]

Define \(\chi(x):=-\tfrac12\log(1-x)\) on \([0,\delta^{2}]\). Then, \(\chi\) is \(C^{3}\) with
\[
\chi(0)=0,\quad \chi'(x)=\tfrac12(1-x)^{-1},\quad \chi''(x)=\tfrac12(1-x)^{-2},\quad 
\chi'''(x)=(1-x)^{-3}.
\]
Taylor at \(0\) gives
\[
\chi(x)=\tfrac12 x+\tfrac14 x^{2}+R_\chi(x),\qquad 
|R_\chi(x)|\le \frac{\sup_{[0,\delta^{2}]}\chi''' }{6}\,x^{3}
\le \frac{(1-\delta^{2})^{-3}}{6}\,x^{3}.
\]
Now take \(x=s^{2}\). Then
\[
\psi(s)=\chi(s^{2})=\tfrac12 s^{2}+\tfrac14 s^{4}+R_\chi(s^{2}),\qquad 
|R_\chi(s^{2})|\le C_2\,|s|^{6},\quad C_2:=\tfrac16(1-\delta^{2})^{-3}.
\]

Summing \(\phi\) and \(\psi\),
\[
\log r(s)=\bigl(-s-\tfrac12 s^{2}+R_\phi(s)\bigr)+\bigl(\tfrac12 s^{2}+\tfrac14 s^{4}+R_\chi(s^{2})\bigr)
=-s+\bigl(R_\phi(s)+R_\chi(s^{2})+\tfrac14 s^{4}\bigr).
\]
Hence
\[
\bigl|\log r(s)+s\bigr|\ \le\ |R_\phi(s)|+|R_\chi(s^{2})|+\tfrac14|s|^{4}
\ \le\ C_1|s|^{3} + C_2|s|^{6} + \tfrac14|s|^{4}
\ \le\ K\,|s|^{3},
\]
for all \(|s|\le \delta\), with \(K:=C_1+\tfrac14\delta+ C_2\delta^{3}\).
\end{proof}

\begin{lemma}\label{L_equiv}
Let $L$ be slowly varying at $\infty$. Then, as $s\downarrow 0$,
\[
\frac{L\!\Big(\dfrac{1}{\,1-\sqrt{\,1-s^{2}\,}}\Big)}{L\!\Big(\dfrac{1}{s^{2}}\Big)}\ \longrightarrow\ 1.
\]
\end{lemma}

\begin{proof}
For $s\in(0,1)$,
\[
1-\sqrt{1-s^{2}}
=\frac{(1-\sqrt{1-s^{2}})(1+\sqrt{1-s^{2}})}{1+\sqrt{1-s^{2}}}
=\frac{s^{2}}{\,1+\sqrt{1-s^{2}}\,}.
\]
Hence
\[
\frac{1}{1-\sqrt{1-s^{2}}}
=\frac{1+\sqrt{1-s^{2}}}{s^{2}}
=\lambda_s\,x_s,
\qquad
x_s:=\frac{1}{s^{2}}\xrightarrow[s\downarrow 0]{}\infty,\quad
\lambda_s:=1+\sqrt{1-s^{2}}\in(1,2).
\]

Since $L$ is slowly varying at $\infty$, the Uniform Convergence Theorem
(see, e.g., \cite[Thm.~1.2.1]{BGT1987}) yields that for every compact
$K\subset(0,\infty)$,
\[
\sup_{\lambda\in K}\Bigg|\frac{L(\lambda x)}{L(x)}-1\Bigg|\xrightarrow[x\to\infty]{}0.
\]
Apply this with $K=[1,2]$, $x=x_s$ and $\lambda=\lambda_s$ to obtain:
for every $\varepsilon>0$ there exists $\delta_\varepsilon\in(0,1)$ such that for all
$s\in(0,\delta_\varepsilon)$,
\[
\left|\frac{L(\lambda_s x_s)}{L(x_s)}-1\right|
=\left|\frac{L\!\big(\tfrac{1}{\,1-\sqrt{1-s^{2}}}\big)}{L\!\big(\tfrac{1}{s^{2}}\big)}-1\right|
<\varepsilon.
\]
This proves the claim. \qedhere

\end{proof}

\begin{remark}
 Lemma~\ref{L_equiv}. Fix $\varepsilon\in(0,1)$. Then, there exists
$\delta_\varepsilon\in(0,1)$ such that for all $s\in(0,\delta_\varepsilon)$,
\begin{equation}\label{epsilon_L_2}
(1-\varepsilon)\,L\!\Big(\frac{1}{s^{2}}\Big)
\;\le\;
L\!\Big(\frac{1}{\,1-\sqrt{\,1-s^{2}\,}}\Big)
\;\le\;
(1+\varepsilon)\,L\!\Big(\frac{1}{s^{2}}\Big).
\end{equation}

\end{remark}
\begin{lemma}\label{lem:Potter}
Let $L$ be slowly varying at $\infty$. Then, for every $\varepsilon>0$ and every $A>1$ there exist
$X=X(A,\varepsilon)\ge 1$ and $n_0=n_0(A,\varepsilon)\in\mathbb{N}$ such that, for all $n\ge n_0$ and all $t>0$ with
\[
n^{2}\ge X
\quad\text{and}\quad
\frac{n^{2}}{t^{2}}\ge X
\ \ \big(\text{equivalently }0<t\le n/\sqrt{X}\big),
\]
\begin{equation}\label{eq:Potter-global}
A^{-1}\,\min\{t^{2\varepsilon},\,t^{-2\varepsilon}\}\,L(n^{2})
\;\le\;
L\!\Big(\frac{n^{2}}{t^{2}}\Big)
\;\le\;
A\,\max\{t^{2\varepsilon},\,t^{-2\varepsilon}\}\,L(n^{2}).
\end{equation}
In particular, if $t\in(0,1]$, one obtains the one–sided form
\begin{equation}\label{eq:Potter-I}
A^{-1}\,t^{2\varepsilon}\,L(n^{2})
\;\le\;
L\!\Big(\frac{n^{2}}{t^{2}}\Big)
\;\le\;
A\,t^{-2\varepsilon}\,L(n^{2})
\qquad (n\ \text{sufficiently large}).
\end{equation}
Moreover, for every compact interval $0<a\le t\le b<\infty$,
\begin{equation}\label{eq:UCT-compact}
\sup_{t\in[a,b]}\Bigg|\frac{L\!\big(n^{2}/t^{2}\big)}{L(n^{2})}-1\Bigg|\xrightarrow[n\to\infty]{}0,
\end{equation}
that is, the convergence is uniform on compact sets of \(t\).
\end{lemma}

\begin{proof}
By Potter’s bound for slowly varying functions (regularly varying of index \(0\)),
for every \(\varepsilon>0\) and every \(A>1\) there exists \(X=X(A,\varepsilon)\ge 1\) such that, for all
\(x,y\ge X\),
\begin{equation}\label{eq:Potter-master}
A^{-1}\,\min\!\Big\{\Big(\frac{y}{x}\Big)^{\varepsilon},\Big(\frac{y}{x}\Big)^{-\varepsilon}\Big\}
\;\le\;
\frac{L(y)}{L(x)}
\;\le\;
A\,\max\!\Big\{\Big(\frac{y}{x}\Big)^{\varepsilon},\Big(\frac{y}{x}\Big)^{-\varepsilon}\Big\}.
\end{equation}
(See \cite[Thm.~1.5.6]{BGT1987}).

Apply \eqref{eq:Potter-master} with \(x=n^{2}\) and \(y=n^{2}/t^{2}\). The stated size
conditions ensure \(x,y\ge X(A,\varepsilon)\), and we obtain \eqref{eq:Potter-global}. 
If, in addition, \(t\in(0,1]\), then \(\min\{t^{2\varepsilon},t^{-2\varepsilon}\}=t^{2\varepsilon}\) and
\(\max\{t^{2\varepsilon},t^{-2\varepsilon}\}=t^{-2\varepsilon}\), which yields \eqref{eq:Potter-I} once \(n\) is large
enough so that \(n^{2}\ge X(A,\varepsilon)\).

Finally, \eqref{eq:UCT-compact} follows from the Uniform Convergence Theorem for slowly varying functions (see, e.g., \cite[Thm.~1.2.1]{BGT1987}). This completes the proof.
\end{proof}

\begin{lemma}\label{lem:L_osc}
Let
\[
L(x):=(\log x)^{\sin(\log\log x)}=\exp\!\big((\log\log x)\,\sin(\log\log x)\big),\qquad x>e.
\]
Then, \(L\) is a slowly varying function with
\[
\liminf_{x\to\infty} L(x)=0
\qquad\text{and}\qquad
\limsup_{x\to\infty} L(x)=\infty.
\]
\end{lemma}

\begin{proof}

Fix $\lambda>0$ and set $u(x):=\log\log x$ and
\[
\delta(x):=u(\lambda x)-u(x)
=\log\!\left(\frac{\log(\lambda x)}{\log x}\right)
=\log\!\left(1+\frac{\log\lambda}{\log x}\right).
\]
Since $\frac{\log\lambda}{\log x}\to 0$ as $x\to\infty$, there exists $x_0$ such that for $x\ge x_0$,
\[
|\delta(x)|\le \frac{2|\log\lambda|}{\log x}\leq 1.
\]
Write $\log L(x)=g(u(x))$ with $g(u):=u\sin u$. By the mean value theorem,
\[
\big|g(u(x)+\delta(x))-g(u(x))\big|
=|g'(\xi_x)|\,|\delta(x)|,
\]
for some $\xi_x$ between $u(x)$ and $u(x)+\delta(x)$, where
\[
g'(u)=\sin u + u\cos u,
\qquad\text{so}\qquad
|g'(\xi_x)|\le 1+\xi_x\le 2+u(x)\quad\text{for $x\ge x_0$}.
\]
Hence, for $x\ge x_0$,
\[
\left|\log\frac{L(\lambda x)}{L(x)}\right|
=\big|g(u(\lambda x))-g(u(x))\big|
\le \big(2+\log\log x\big)\,\frac{2|\log\lambda|}{\log x}
\xrightarrow[x\to\infty]{} 0.
\]
Therefore, $\frac{L(\lambda x)}{L(x)}\to 1$, i.e., $L$ is slowly varying at infinity.

For $k\in\mathbb N$ set
\[
t_k:=\frac{\pi}{2}+2\pi k,\quad s_k:=\frac{3\pi}{2}+2\pi k,\quad
x_k:=\exp(\exp t_k),\quad y_k:=\exp(\exp s_k).
\]
Then, $u(x_k)=t_k$ and $u(y_k)=s_k$, hence
\[
L(x_k)=\exp\!\big(t_k\sin t_k\big)=\exp(t_k)\xrightarrow[k\to\infty]{}\infty,
\quad
L(y_k)=\exp\!\big(s_k\sin s_k\big)=\exp(-s_k)\xrightarrow[k\to\infty]{}0,
\]
which proves the claim.
\end{proof}

\begin{lemma}\label{lem:G_limits}
Let 
\[
G_n\ :=\ n\,\mathbb P(D^{\rightarrow}\ge n)\,\log\!\Big(\frac{1}{\mathbb P(D^{\rightarrow}\ge n)}\Big),\qquad n\in\mathbb N.
\]
Assume there exist $0<A_- \le A_+<\infty$ such that
\begin{equation}\label{eq:scale_comp}
0<A_-\ \le\ \liminf_{n\to\infty}\frac{n\,\mathbb P(D^{\rightarrow}\ge n)}{L(n^2)}
\ \le\ A_+\ <\infty,
\end{equation}
where $L(x)$ is as defined in Lemma~\ref{lem:L_osc}. Then:
\begin{align*}
\frac{A_-}{2}\ \le\ \liminf_{n\to\infty} G_n.
\end{align*}
\end{lemma}

\begin{proof}
Fix $n\ge 2$ and set
\[
\ell_n:=\log(n^2)=2\log n,\qquad
u_n:=\log\ell_n=\log(2\log n),\qquad
t_n:=\sin u_n\in[-1,1].
\]
Then, $L(n^2)=e^{t_n u_n}$. Define
\[
H_n\ :=\ -\log\!\Big(\frac{L(n^2)}{n}\Big)\,L(n^2)
= e^{t_n u_n}\big(\log n- t_n u_n\big).
\]
A direct identity gives
\begin{equation}\label{eq:G_decomp_final}
G_n
= a_n\,H_n\ -\ a_n\,L(n^2)\,\log a_n,
\qquad
a_n:=\frac{n\,\mathbb P(D^\rightarrow\ge n)}{L(n^2)}.
\end{equation}

From \eqref{eq:scale_comp}, for every $\varepsilon>0$ there exists $N_\varepsilon$ such that, for $n\ge N_\varepsilon$,
\begin{equation}\label{eq:a-_a+}
a_n \ \le\ A_+ + \varepsilon
\quad\text{and}\quad
|\log a_n|\le C_\varepsilon,
\end{equation}
for some finite $C_\varepsilon$.

For $n\in\mathbb{N}$, define
\[
h_n(t):=e^{t u_n}\big(\log n - t u_n - \log a_n\big),\qquad t\in[-1,1].
\]
Then
\[
h'_n(t)=u_n e^{t u_n}\,\big(\log n - t u_n - 1 - \log a_n\big).
\]
Since $u_n=\log(2\log n)=o(\log n)$ and $|\log a_n|\le C_\varepsilon$, there is $N_\varepsilon'$ such that, for $n\ge N_\varepsilon'$ and $t\in[-1,1]$,
\[
\log n - t u_n - 1 - \log a_n\ \ge\ \log n - u_n - 1 - C_\varepsilon\ >\ 0.
\]
Thus, $t\mapsto h_n(t)$ is strictly increasing on $[-1,1]$ for all $n$ large, and hence
\begin{equation}\label{eq:hn_lb_final}
h_n(t_n)\ \ge\ h_n(-1)=e^{-u_n}\big(\log n + u_n - \log a_n\big).
\end{equation}

Since $G_n=a_n h_n(t_n)$ by \eqref{eq:G_decomp_final}, \eqref{eq:hn_lb_final} yields
\[
G_n\ \ge\ a_n\,e^{-u_n}\big(\log n + u_n - \log a_n\big).
\]
Using $e^{u_n}=2\log n$ and $u_n/\log n\to 0$,
\[
\frac{\log n + u_n}{e^{u_n}}=\frac12+\frac{u_n}{2\log n}=\frac12+o(1).
\]
Moreover, from \eqref{eq:a-_a+}, we obtain
\[
\frac{|a_n\,\log a_n|}{e^{u_n}}\ \le\ \frac{(A_++\varepsilon)\,C_\varepsilon}{2\log n}\ \xrightarrow[n\to\infty]{} 0.
\]
Hence
\[
\liminf_{n\to\infty} G_n\ \ge\ \frac12\,\liminf_{n\to\infty} a_n\ \ge\ \frac{A_-}{2}.
\]

\end{proof}




\begin{acks}
G.O.C. was supported by the Coordenação de Aperfeiçoamento de Pessoal de Nível Superior - Brasil (CAPES) - Finance Code 001. F.P.M. and J.H.R.G. are  supported by FAPESP (grants 2023/13453-5 and 2025/03804-0, respectively).
\end{acks}


\end{document}